%% file: main.tex
\documentclass{amsart}
\usepackage{ifthen}
\usepackage{amssymb}
\usepackage{accents}
\usepackage{enumerate}
\usepackage[centertags]{amsmath}

\newcommand{\eps}{\varepsilon}
\numberwithin{equation}{section}

\newcommand{\bR}{\mathbb{R}}

\newcommand{\be}{\begin{equation}}
\newcommand{\ee}{\end{equation}}
\newcommand{\bea}{\begin{eqnarray}}
\newcommand{\eea}{\end{eqnarray}}

\newcommand{\Ric}{\operatorname{Ric}}

\newcommand{\p}{\partial}
\newcommand{\ra}{\rightarrow}
\newcommand{\px}[1]{\frac{\p}{\p x^{#1}}}
\newcommand{\nd}[1]{\frac{1}{#1}}

\newcommand{\Acirc}{\accentset{\circ}{A}}
\newcommand{\Scal}{\operatorname{Sc}}
\newcommand{\R}{\operatorname{R}}
\newcommand{\CW}{\mathcal{W}}
\newcommand{\CF}{\mathcal{F}}
\newcommand{\dmu}{\,\mathrm d \mu}

\newcommand{\dr}{\left.\tfrac{\partial}{\partial r}\right|_{r=0}}
\newcommand{\drr}{\left.\tfrac{\partial^2}{\partial r^2}\right|_{r=0}}
\newcommand{\dd}[1]{\tfrac{\partial}{\partial #1}}

\newtheorem{definition}{Definition}[section]

\newtheorem{theorem}[definition]{Theorem}
\newtheorem{proposition}[definition]{Proposition}
\newtheorem{lemma}[definition]{Lemma}
\newtheorem{corollary}[definition]{Corollary}

\theoremstyle{remark}

\makeatletter
\begin{document}
\title{Local foliation of manifolds by surfaces of Willmore type}
\author{Tobias Lamm}
\address{Tobias Lamm: Karlsruhe Institute of Technology, Englerstrasse 2,
76131 Karlsruhe, Germany.}
\email{tobias.lamm@kit.edu}
\author{Jan Metzger}
\address{Jan Metzger: University of Potsdam, Institute for
  Mathematics, Karl-Liebknecht-Stra\ss{}e 24/25, 14476 Potsdam, Germany}
\email{jan.metzger@uni-potsdam.de}
\author{Felix Schulze}
\address{Felix Schulze: 
  Department of Mathematics, University College London, 25 Gordon St,
  London WC1E 6BT, UK}
\email{f.schulze@ucl.ac.uk}
\date{\today}
\begin{abstract}
  We show the existence of a local foliation of a three dimensional
  Riemannian manifold by critical points of the Willmore functional
  subject to a small area constraint around non-degenerate critical
  points of the scalar curvature. This adapts a method developed by
  Rugang Ye to construct foliations by surfaces of constant mean
  curvature.
\end{abstract}
\maketitle

\input{intro}

\input{expansion-geometric}

\input{equation}
\input{foliation}
\input{uniqueness}

\bibliographystyle{abbrv}
\bibliography{localfoliation}

\end{document}

%% file: intro.tex

\section{Introduction}
\label{sec:introduction}
In this paper we consider the Willmore functional 
\begin{equation*}
  \CF(\Sigma) = \frac{1}{4} \int_\Sigma H^2 \dmu
\end{equation*}
for surfaces $\Sigma$ immersed in a three dimensional Riemannian
manifold $(M,g)$. Here $H=\lambda_1+\lambda_2$ denotes the sum of the
principal curvatures of $\Sigma$.

More precisely, we consider the variational problem
\begin{equation}
  \label{eq:minprob}
  \inf \{ \CF(\Sigma) \mid \Sigma \hookrightarrow M \text{ with }
  |\Sigma| = a \}
\end{equation}
where $a\in (0,\infty)$ is a (small) prescribed constant and
$|\Sigma|$ denotes the area of $\Sigma$ with respect to the induced
metric.

The Euler-Lagrange equation for this variational problem is
\begin{equation}
  \label{eq:maineq}
  \Delta H + H|\Acirc|^2 + H \Ric(\nu,\nu) = \lambda H.
\end{equation}
Here $\Delta$ denotes the Laplace-Beltrami operator on $\Sigma$,
$\Acirc$ is the trace free part of the second fundamental form $A$ of
$\Sigma$ and $\Ric(\nu,\nu)$ is the Ricci-curvature of $(M,g)$ in
direction of the normal $\nu$ to $\Sigma$. Note that the left hand
side is two times the first variation of $\CF$ and the right hand side of this
expression is a Lagrange-parameter $\lambda\in\R$ multiplied with the first
variation of the area functional.

In previous papers the first two authors have shown that if $(M,g)$ is
compact then there exists a small $a_0\in(0,\infty)$ depending only on
$(M,g)$ such that the infimum in \eqref{eq:minprob} is attained for
all $a\in(0,a_0)$ on smooth surfaces $\Sigma_a$ \cite{LM:2013}. See
\cite{Chen-Li:2014} and \cite{Mondino-Riviere:2013} for alternative
proofs and \cite{Mattuschka} for a recent parabolic approach.
Existence and multiplicity results of Willmore surfaces in Riemannian
manifolds have been studied previously in a perturbative setting in
\cite{Mondino2010}, \cite{Mondino2013} where the functionals $\CF$ and
the $L^2$-norm of $\Acirc$ are considered without a constraint.

For $a\to 0$ the surfaces $\Sigma_a$ converge to critical points of
the scalar curvature~\cite{LM:2010,LM:2013,Laurin-Mondino:2014}. A
similar result has been obtained previously for small isoperimetric
surfaces by Druet~\cite{Druet2002}. This  was later generalized
by Laurain~\cite{Laurain2012} to surfaces with constant mean
curvature.

It is natural to ask about the precise structure of this family
when $a$ tends to zero.  A similar situation was considered by
Ye~\cite{Ye:1991} for surfaces of constant mean curvature, which are
critical for the isoperimetric problem, that is to minimize area
subject to prescribed enclosed volume. He proves that given a
non-degenerate critical point $p$ of the scalar curvature one can find
a pointed neighborhood $\dot U = U\setminus\{p\}$ which is foliated by
hyper-surfaces of constant mean curvature. That is
$\dot U = \bigcup_{H\in(H_0,\infty)} \Sigma_H$ where $\Sigma_H$ has
constant mean curvature $H$. For $H\to \infty$ these surfaces become
spherical and approach geodesic spheres $S_r(p)$ with radius
$r\approx \frac{2}{H}$. Ye uses an implicit function argument to show
that the $\Sigma_H$ can be constructed as graphs over $S_r(0)$. The
main difficulty is that the operator linearizing the mean curvature
has an approximate kernel corresponding to translations. This
approximate kernel can be dealt with by allowing a translation of the
$S_r(0)$ and using the non-degeneracy of the second derivative of the
scalar curvature. Our result in this paper is to adapt the method of
Ye to the case of the Willmore functional. More precisely, we get the
following result:
\begin{theorem}
  \label{thm:main}
  Let $(M,g)$ be a smooth Riemannian manifold and let $p\in M$ be such
  that $\nabla \Scal (p) = 0$ and such that $\nabla^2 \Scal (p)$ is
  non-degenerate. Then there exists $a_0\in(0,\infty)$, a neighborhood
  $U$ of $p$ and for each $a\in(0,a_0)$ a spherical surface $\Sigma_a$
  which satisfies \eqref{eq:maineq} for some $\lambda\in \R$ and 
  $|\Sigma_a|=a$. The $\Sigma_a$ are mutually disjoint and
  $\bigcup_{(0,a_0)} \Sigma_a = U\setminus\{p\}$.
\end{theorem}
More detailed information on the structure of this foliation can be
found in section~\ref{sec:equation} where the implicit function
argument is carried out. In particular, we refer to
Corollary~\ref{thm:uniqueness2} for some comments about the local
uniqueness of the $\Sigma_a$.

The paper is organized as follows: In section~\ref{sec:expansion} we
calculate the expansion of the Willmore functional on small geodesic
spheres to set up the argument. In section~\ref{sec:equation} we use
the implicit function theorem to solve the equation in a very similar
manner to Ye. First we solve the equation in the kernel of the
linearized operator using the non-degeneracy condition on the scalar
curvature and by a generic implicit function argument we solve
perpendicular to the kernel. Proposition~\ref{thm:is-foliation} in
section~\ref{sec:foliation} establishes that the $\Sigma_a$ indeed
form a foliation as claimed. Finally in section~\ref{sec:uniqueness}
we prove a local uniqueness reslut for the $\Sigma_a$ as solutions
to~\eqref{eq:maineq}.

\subsection*{Acknowlegements}
The first and second author were supported by the DFG with grants LA
3444/1-1 resp. ME 3816/1-2. 
\\
The authors would like to thank the referee for the careful reading of the paper.

\subsection*{Remarks}
During the preparation of this manuscript, the authors learned that
Norihisa Ikoma, Andrea Malchiodi and Andrea Mondino
\cite{Ikoma-Malchiodi-Mondino} have an independent proof of
Theorem~\ref{thm:main}.

%% file: expansion-geometric.tex

\section{The Willmore operator on geodesic spheres}
\label{sec:expansion}
In this section we compute the basic geometric quantities and the Willmore operator of small geodesic spheres. We consider a setup similar as in \cite{Ye:1991}, i.e. we  consider a point $p\in M^3$ and an orthonormal basis $\{e_j\}_{j=1}^3$ of  $T_pM$ which we use to identify $T_pM$ with $\bR^3$. Furthermore, we consider the map
$$ \phi:\mathbb{R}^{3}\supset B_{\rho_p}(0) \ra M: x \mapsto \exp_{p}(x^ie_i)\ ,$$
where $\rho_{p}>0$ is the injectivity radius of $p$.
Let $\tilde{g}$ be the pulled back metric of $M$ via $\phi$, with
$\langle \cdot,\cdot\rangle$ denoting the euclidean metric on
$\mathbb{R}^3$. We consider the map $\Psi_\sigma: \bR^3\ra \bR^3: x
\mapsto \sigma x$ and
denote $g := \sigma^{-2}\Psi_\sigma^*\tilde{g}$. 

We now compute the
second fundamental form of $S_\rho(0) \subset T_pM$ for $0<\rho<\sigma^{-1} \rho_{p}$. The normal to $S_\rho$ w.r.t. $g$ is given by $x/|x|$, and working w.l.o.g.~at the north pole, i.e. $e_i$ for $i=1,2$ are tangent vectors and $e_{3}$ is parallel to the normal, we obtain
\begin{align*}
 h_{ij} &= g\left(e_i, \nabla_{e_j}\nu\right) = g\left(e_i, \nabla_{e_j}\frac{x^l}{|x|}e_l\right)
	= g\left(e_i, \left(\frac{\delta^{jl}}{\rho}-\frac{x^lx^j}{\rho^3}\right)e_l\right)+\frac{x^l}{\rho}g\left(e_i,\Gamma_{jl}^k e_k\right)\\
&= \frac{1}{\rho}\left(g_{ij}-\frac{x^lx^j}{\rho^2}g_{il}+x^l\Gamma_{jl}^k g_{ik}\right)\ .
\end{align*} 
This yields since $x^l=0$ for $ l=1,2$
\begin{equation}\label{eq:2.1}
 h^i_{\,j} = \frac{1}{\rho}\Big(\delta^i_{\,j}+x^l\Gamma_{lj}^i\Big)\ .
\end{equation}
Furthermore, we have
\begin{align*}
0 = g\big(\nu, p^\perp_{\nu}(e_i)\big)= \frac{x^k}{\rho}g_{ki}-\frac{x^i}{\rho} \ ,
\end{align*} 
where $p^\perp_{\nu}(\cdot)=e-\langle e_i,\nu\rangle \nu$ is the orthogonal projection onto the subspace perpendicular to $\nu$. This gives
\begin{equation}\label{eq:2.2}
\delta_{im}=\frac{\p}{\p x^m}\left(x^kg_{ki}\right)=g_{mi}+x^k\frac{\p}{\p x^m}g_{ki} ,
\end{equation}
which we can use to compute
\begin{equation}\label{eq:2.3}
\begin{split}
 x^l\Gamma_{lj}^i &= \nd{2}g^{ik}\Big(x^l\px{l}g_{jk}+x^l\px{j}g_{lk}-x^l\px{k}g_{lj}\Big)\\
		  &= \nd{2}g^{ik}\Big(x^l\px{l}g_{jk}+\delta_{jk}-g_{jk}-\delta_{kj}+g_{kj}\Big)\\
		  &= \nd{2}g^{ik}x^l\px{l}g_{jk}\ .
\end{split}
\end{equation}
We denote partial
derivatives with a semicolon, instead of a comma for covariant
derivatives. From \cite{Lee-Parker:1987} and the definition of $g$ we have the formula
\begin{equation}\label{eq:2.4}
\begin{split}
g_{ij}(x)= &\ g_{ij}(0)+\frac{\sigma^2}{3}\R_{ipqj} x^px^q+\frac{\sigma^3}{6}\R_{ipqj,r} x^px^qx^r\\
		&+\sigma^4\Big(\nd{20}\R_{ipqj,rs} +\frac{2}{45}\R_{ipqt} \R_{jrst} \Big)x^px^qx^rx^s + O(\sigma^5|x|^5)\, ,
\end{split}
\end{equation}
where the curvature terms are all corresponding to $g$ and are evaluated at $0$. Since 
$$\px{m}g^{ij}= -g^{iv}g^{jw}\px{m}g_{vw}$$
and differentiating further (using that the first derivatives of $g_{ij}$ vanish at $0$) we obtain:
\begin{equation}\label{eq:2.5}
g^{ij}(x)= g^{ij}(0)- \frac{\sigma^2}{3}\R^{i\ \ j}_{\ pq
  \ } x^px^q - \frac{\sigma^3}{6}\R_{\ pq\, ,r}^{i\ \ j} x^px^qx^r + O(\sigma^4|x|^4) \ .
\end{equation}
Combining \eqref{eq:2.4} and \eqref{eq:2.5}, we see 
\begin{equation}\label{eq:2.6}
\begin{split}
\nd{2}g^{ik}x^l\px{l}g_{jk} = &\ \frac{1}{2} g^{ik} \frac{\p}{\p s}\left(g_{kj}(sx)\right)\Big|_{s=1}\\
 = &\ g^{ik}\Big(\frac{\sigma^2}{3}
\R_{kpqj} x^px^q+\frac{\sigma^3}{4} \R_{kpqj,r} x^px^qx^r+\sigma^4\Big(\nd{10}\R_{kpqj,rs} \\ 
&+\frac{4}{45}\R_{kpqt} \R_{jrst} \Big)x^px^qx^rx^s + O(\sigma^5|x|^5)\Big)\\
= & \ \frac{1}{3}\sigma^2
\R_{\ pqj}^i x^px^q+\nd{4}\sigma^3 \R_{\ pqj,r}^i
x^px^qx^r+\sigma^4\Big(\nd{10}\R_{\ pqj,rs}^i \\ 
&-\frac{1}{45}\R_{\ pqt}^i \R_{jrst} \Big)x^px^qx^rx^s + O(\sigma^5|x|^5)\ .
\end{split}
\end{equation}
Combining this with \eqref{eq:2.1} and \eqref{eq:2.3}, we can thus write
$$\rho\cdot h^i_{\ j} = \delta^i_{\ j}+\nd{2}g^{ik}x^l\px{l}g_{ji}=
\exp(a^i_{\ j})$$
where $\exp$ here is the exponential map on matrices and 
\begin{align*}
a^i_{\ j}:=& \ \frac{1}{3}\sigma^2
\R_{\ pqj}^i x^px^q+\nd{4}\sigma^3 \R_{\ pqj,r}^i
x^px^qx^r+\sigma^4\Big(\nd{10}\R_{\ pqj,rs}^i \\ 
& \ -\frac{7}{90}\R_{\ pqt}^i \R_{jrst} \Big)x^px^qx^rx^s + O(\sigma^5|x|^5)\ .
\end{align*}
This yields 
\begin{equation*}
\begin{split}
\det\nolimits_{3}&\Big( \delta^i_{
  j}+\nd{2}g^{ik}x^l\px{l}g_{jk}\Big)= \ \exp(\text{tr}(a^i_{ j}))\\
&= \ 1- \frac{1}{3}\sigma^2
\R_{pq} x^px^q-\nd{4}\sigma^3 \R_{pq,r} x^px^qx^r+\sigma^4\Big(-\nd{10}\R_{pq,rs} \\ 
& \ \ \ \ -\frac{7}{90}\R_{\ pqt}^k \R_{krst}+\nd{18}\R_{pq}\R_{rs} \Big)x^px^qx^rx^s + O(\sigma^5|x|^5)\ .
\end{split}
\end{equation*}
Note that for the Gauss curvature we have from \eqref{eq:2.1}
\begin{equation*}
 K_{S_\rho}= \nd{\rho^{2}} \det\nolimits_{2}\Big( \delta^i_{ j}+\nd{2}g^{ik}x^l\px{l}g_{jk}\Big)
= \nd{\rho^{2}} \det\nolimits_{3}\Big( \delta^i_{ j}+\frac{r}{2}g^{ik}\frac{\p}{\p r}g_{kj}\Big)\ ,
\end{equation*}
since $\frac{\p}{\p r}g_{k3}= 0\ \ \forall\, k=1,2,3$. Combining this with the above computation, this yields
\begin{equation}
\begin{split} \label{eq:2.7}
K_{S_\rho}(x)=\ & \nd{\rho^{2}}\Big( 1- \frac{1}{3}\sigma^2
\Ric_{pq} x^px^q-\nd{4}\sigma^3 \Ric_{pq,r} x^px^qx^r+\sigma^4\Big(-\nd{10}\Ric_{pq,rs} \\ 
& \ -\frac{7}{90}\R_{\ pqt}^k \R_{krst}+\nd{18}\Ric_{pq}\Ric_{rs} \Big)x^px^qx^rx^s\Big) + \rho^{-2}O(\sigma^5|x|^5)\ .
\end{split}
\end{equation}
Combining \eqref{eq:2.1}, \eqref{eq:2.3} and \eqref{eq:2.6} we get for the mean curvature, using $\Ric_{pq} = - \R^i_{\ pqi}$,
\begin{equation}
\begin{split} \label{eq:2.8}
 H_{S_\rho} =\ & \frac{1}{\rho} \text{tr}_{2}\Big(
  \delta^i_{\ j}+\nd{2}g^{ik}x^l\px{l}g_{jk}\Big) =
  \frac{1}{\rho}\Big(2+\text{tr}_{3}\Big(\nd{2}g^{ik}x^l\px{l}g_{ji}\Big)\Big)\\
=\ & \frac{1}{\rho}\Big(2- \frac{1}{3}\sigma^2
\Ric_{pq} x^px^q-\nd{4}\sigma^3 \Ric_{pq,r} x^px^qx^r-\sigma^4\Big(\nd{10}\Ric_{pq,rs} \\ 
&+\frac{1}{45}\R_{\ pqt}^i \R_{irst} \Big)x^px^qx^rx^s\Big) + \rho^{-1}O(\sigma^5|x|^5)\ .
\end{split}
\end{equation}
For the norm squared of the traceless second fundamental form, we have $|\Acirc|^2=\nd{2}(H^2-4K)$, which yields
\begin{equation}\label{eq:2.9}
  |\Acirc|^2= \rho^{-2}\sigma^4\Big(\frac{1}{9}\R^{i\ \ t}_{\
    pq}\R_{irst}
  -\nd{18}\Ric_{pq}\Ric_{rs}\Big)x^px^qx^rx^s + \rho^{-2}O(\sigma^5|x|^5)
\end{equation}
We now aim to compute the laplacian of $H$ on $S_1$. Using formula $(3.2)$ in \cite{Ecker04}, we have
\begin{equation}
\begin{split} \label{eq:2.10}
\Delta^{S_1}H=\ &
\Delta^g\tilde{H}-\text{Hess}(\tilde{H})(\nu,\nu)+g(\nabla \tilde{H},
\vec{H})\\
=\ &  \Delta^g\tilde{H}-\text{Hess}(\tilde{H})(x,x) - H x^i\frac{\partial
  \tilde{H}}{\partial x^i},
\end{split}
\end{equation}
where $\vec{H} = - H \nu$ is the mean curvature vector of $S_1$, and $\tilde{H}$ is any extension of $H$. We have
\begin{equation} \label{eq:2.11}
\Delta^g\tilde{H} = g^{ij}\frac{\p^2\tilde{H}}{\p x^i\p x^j} -
g^{ij}\Gamma_{ij}^k\frac{\p \tilde H}{\p x^k}\ .
\end{equation}
From \eqref{eq:2.4} and \eqref{eq:2.5} we get
\begin{equation}
\begin{split} \label{eq:2.12}
g^{ij}\Gamma_{ij}^k &= \nd{2}g^{ij}g^{kl}\Big(\frac{\p}{\p
  x^i}g_{lj}+ \frac{\p}{\p x^j}g_{il}-\frac{\p}{\p x^l}g_{ij}\Big) \\
  &=
g^{kl}g^{ij}\frac{\p}{\p x^i}g_{lj} - \nd{2}g^{kl}g^{ij}\frac{\p}{\p
  x^l}g_{ij}\\
&= \Big(g_\tau^{kl} + O(\sigma^2)\Big)\Big(g_\tau^{ij} +
O(\sigma^2)\Big)
\Big(\frac{\sigma^2}{3}\Big(\R_{lipj}+\R_{lpij}\Big)x^p +
O(\sigma^3)\Big)\\
&\ \ \  - \nd{2}\Big(g_\tau^{kl} + O(\sigma^2)\Big)\Big(g_\tau^{ij} +
O(\sigma^2)\Big)\Big(\frac{\sigma^2}{3}\Big(\R_{ilpj}+\R_{iplj}\Big)x^p +
O(\sigma^3)\Big)\\
&= \frac{2}{3}\sigma^2 \Ric^k_{\ p}x^p + O(\sigma^3)\, .
\end{split}
\end{equation}
We choose an extension of the mean curvature $H$ on $S_1$ via
$$\tilde{H}:= \rho H_{S_\rho}\, .$$
Combining this with \eqref{eq:2.8} and \eqref{eq:2.12}, this yields
\begin{equation}\label{eq:2.13}
  - g^{ij}\Gamma_{ij}^k\frac{\p \tilde H}{\p x^k} =
  \frac{4}{9}\sigma^4 \Ric^{\ k}_{p}\Ric_{kq}x^px^q +
  O(\sigma^5)\ .
\end{equation}
Similarly, using \eqref{eq:2.5} and \eqref{eq:2.8} we obtain
\begin{equation}
\begin{split} \label{eq:2.14}
   g^{ij}\frac{\p^2\tilde{H}}{\p x^i\p x^j} =&\ \Big(g^{ij} -
   \frac{\sigma^2}{3}\R^{i\ \
     j}_{pq}x^px^q+O(\sigma^3)\Big)\\
&\ \cdot \bigg(-\frac{2}{3}\sigma^2\Ric_{ij} -
\frac{\sigma^3}{2}\big(\Ric_{ij,p}+2\Ric_{ip,j}\big)x^p\\
&\ -
\sigma^4
\bigg(\nd{5}\Ric_{ij,pq}+\frac{2}{5}\Ric_{ip,jq}+\frac{2}{5}\Ric_{ip,qj}+\nd{5}\Ric_{pq,ij}\\
&\ +\frac{4}{45}\R^{s\
  \ t}_{\ ij}\R_{spqt}+\frac{8}{45}\R^{s\
  \ t}_{\ ip}\R_{sjqt}\bigg)x^px^q+ O(\sigma^5)\bigg)\\
=&\ -\frac{2}{3}\sigma^2\Scal-\sigma^3\Scal_{,p}x^p\\
&\ -
\sigma^4\bigg(\frac{3}{5}\Scal_{,pq}+\nd{5}(\Delta\Ric)_{pq}+\frac{2}{5}\Ric_p^{\
  s}\Ric_{sq}\\
&\ +\frac{4}{45}\Ric^{st}\R_{spqt}+\frac{8}{45}\R^{si
  \ t}_{\ \ p}\R_{siqt}\bigg)x^px^q + O(\sigma^5)\, ,
\end{split}
\end{equation}
where we used that, due to the sign convention on the curvature tensor
$(\Ric_{ij}= -\R^{t}_{\ ijt})$ and the second contracted Bianchi identity $2\Ric^t_{\ i,t}=\Scal_{ ,i}$,
we have 
$$\Ric_{ip,qj}=\Ric_{ip,jq}+\R^{\ \ s}_{qj\ i}\Ric_{sp}+R^{\ \ s}_{qj\
  p}\Ric_{si}$$
and thus
$$g^{ij}\Ric_{ip,qj}=\nd{2}\Scal_{,pq}+\Ric^{\ s}_q\Ric_{sp}+\R^{\
  is}_{q\ \ p}\Ric_{si}\, .$$
 This yields, using \eqref{eq:2.11} with \eqref{eq:2.13} and \eqref{eq:2.14} that
 \begin{equation}
 \begin{split}
 \label{eq:2.14b}
 \Delta^g\tilde{H} =&\ -\frac{2}{3}\sigma^2\Scal-\sigma^3\Scal_{,p}x^p\\
&\ -
\sigma^4\bigg(\frac{3}{5}\Scal_{,pq}+\nd{5}(\Delta\Ric)_{pq}-\frac{2}{45}\Ric_p^{\
  s}\Ric_{sq}\\
&\ +\frac{4}{45}\Ric^{st}\R_{spqt}+\frac{8}{45}\R^{si
  \ t}_{\ \ p}\R_{siqt}\bigg)x^px^q + O(\sigma^5)\, .
 \end{split}
\end{equation}
Furthermore, using \eqref{eq:2.8} we see
\begin{equation}
\begin{split} \label{eq:2.15}
  -\frac{\partial\tilde{H}}{\partial x^i \partial x^j}x^ix^j =&\  \frac{2}{3}\sigma^2
\Ric_{pq} x^px^q+\frac{3}{2}\sigma^3 \Ric_{pq,r} x^px^qx^r\\
&+\sigma^4\bigg(\frac{6}{5}\Ric_{pq,rs} 
+\frac{12}{45}\R_{\ pqt}^i \R_{irst} \bigg)x^px^qx^rx^s + 
O(\sigma^5)
\end{split}
\end{equation}
and combining \eqref{eq:2.3} with \eqref{eq:2.6} and \eqref{eq:2.8}
\begin{equation}
\begin{split} \label{eq:2.16}
 x^ix^j\Gamma_{ij}^k\frac{\p \tilde{H}}{\p x^k} =&\
  \nd{2}g^{kl}x^sx^j\frac{\p}{\p x^s}g_{jl}\frac{\p \tilde{H}}{\p
    x^k} \\
=&\ - \frac{2}{9}\sigma^4\R^k_{\ pqr}\Ric_{ks}x^px^qx^sx^r+
O(\sigma^5) = O(\sigma^5)\, ,
\end{split}
\end{equation}
since $\R^k_{\ pqr}\Ric_{ks}x^px^qx^rx^s=0$ by symmetry considerations.
Combining \eqref{eq:2.15} and \eqref{eq:2.16}, this yields
\begin{equation}
\begin{split} \label{eq:2.17}
  -\text{Hess}(\tilde{H})(x,x)=&\ -\frac{\partial\tilde{H}}{\partial x^i \partial x^j}x^ix^j + x^ix^j\Gamma_{ij}^k\frac{\p \tilde{H}}{\p x^k}\\
  =&\ \frac{2}{3}\sigma^2
\Ric_{pq} x^px^q+\frac{3}{2}\sigma^3 \Ric_{pq,r} x^px^qx^r\\
&\ +\sigma^4\left(\frac{6}{5}\Ric_{pq,rs} 
+\frac{12}{45}\R_{\ pqt}^i \R_{irst} \right)x^px^qx^rx^s\\&\ + 
O(\sigma^5) \, .
\end{split}
\end{equation}
By \eqref{eq:2.8} we have
\begin{equation}
\begin{split} \label{eq:2.18}
  -Hx^i\frac{\p \tilde{H}}{\p x^i} =&\
  \frac{4}{3}\sigma^2\Ric_{pq}x^px^q+\frac{3}{2}\sigma^3\Ric_{pq,r}x^px^qx^r +\sigma^4\Big(\frac{4}{5}\Ric_{pq,rs}\\
  &\ +
\frac{8}{45}\R^{i\ \ t}_{\  pq}\R_{irst} - \frac{2}{9}\Ric_{pq}\Ric_{rs}\Big)x^px^qx^rx^s+O(\sigma^5)
\end{split}
\end{equation}
Combining \eqref{eq:2.10} with \eqref{eq:2.14b}, \eqref{eq:2.17} and \eqref{eq:2.18} we arrive at
\begin{equation}
\begin{split} \label{eq:2.19}
\Delta^{S_1}H=\ &  \Delta^g\tilde{H}-\text{Hess}(\tilde{H})(x,x) - H x^i\frac{\partial
  \tilde{H}}{\partial x^i} \\
=\ & -\frac{2}{3}\sigma^2\Scal + 2 \sigma^2 \Ric_{pq}x^px^q - \sigma^3
\Scal_{,p}x^p+3\sigma^3 \Ric_{pq,r}x^px^qx^r\\
& - \sigma^4
\Big(\frac{3}{5}\Scal_{,pq}+\frac{1}{5}\Delta\Ric_{pq}-\frac{2}{45}\Ric_p^{\
  k}\Ric_{kq}+\frac{4}{45}\Ric^{kl}\R_{kpql}\\
&\ +\frac{8}{45}\R^{kl\
  m}_{\ \ p}\R_{klqm}\Big)x^px^q + \sigma^4 \Big(2
\Ric_{pq,rs}+\frac{4}{9}\R^{k\ \ l}_{\ pq}\R_{krsl}\\
&\ -\frac{2}{9}\Ric_{pq}\Ric_{rs}\Big)x^px^qx^rx^s + O(\sigma^5)
\end{split}
\end{equation}
We now aim to compute the area constrained Willmore equation on $S_1$, that is for $\lambda \in \mathbb{R}$ the quantity
\begin{equation}
 \label{eq:2.20}
\mathcal{W}_{\sigma,\lambda}:=\Delta^{S_1}H+ H|\Acirc|^2 + H\Ric(\nu,\nu)+ \sigma^2\lambda H\, .
\end{equation}
To deal with the Ricci term we do a Taylor 
expansion in normal coordinates on the original manifold around $p$.  We get for the Ricci curvature of $\tilde{g}$ that
$$\Ric_{pq}(x) =
\Ric_{pq}(0)+\Ric_{pq;r}(0)x^r+\frac{1}{2}\Ric_{pq;rs}(0)x^rx^s+O(|x|^3)\
.$$
Rescaling as before via the map $\Psi_\sigma$, this implies for the Ricci curvature of $g$
\begin{equation}
  \label{eq:12}
  \Ric_{pq}(x) =
  \sigma^2\Ric_{pq}+\sigma^3
  \Ric_{pq;r}x^r+\frac{\sigma^4}{2}\Ric_{pq;rs}x^rx^s+O(\sigma^5|x|^3)\,
  .
\end{equation}
Recall that we denote partial
derivatives with a semicolon, instead of a comma for covariant
derivatives. Since the Christoffel symbols and derivatives thereof are of order at
least $\sigma^2$ we see that we have on $S^1$:
\begin{align*}
 \Ric(\nu,\nu)= \sigma^2\Ric_{pq}x^px^q+\sigma^3
\Ric_{pq,r}x^px^qx^r+\frac{\sigma^4}{2}\Ric_{pq,rs}x^px^qx^rx^s+O(\sigma^5)\
\end{align*}
 and thus, combining this with \eqref{eq:2.8}
\begin{equation}
\begin{split} \label{eq:2.21}
  H \Ric(\nu,\nu)=&\ 2\sigma^2\Ric_{pq}x^px^q+2\sigma^3
\Ric_{pq,r}x^px^qx^r\\
&\
+\sigma^4\Big(\Ric_{pq,rs}-\frac{1}{3}\Ric_{pq}\Ric_{rs}\Big) x^px^qx^rx^s+O(\sigma^5)\, .
\end{split}
\end{equation}
Combining this with \eqref{eq:2.19}, \eqref{eq:2.8} and \eqref{eq:2.9} we arrive at the following proposition, where we replace the radius $\rho$ by $r$.

\begin{proposition}\label{prop:expansion-Willmore}Let $p\in M^3$ and  $\{e_j\}_{j=1}^3$ be an orthonormal basis of $T_pM$, via which $T_pM$ can be identified with $\bR^3$. Furthermore, we consider the map
$$ \phi:\mathbb{R}^{3}\supset B_{\rho_p}(0) \ra M: x \mapsto \exp_{p}(x^ie_i)\ ,$$
where $\rho_{p}>0$ is the injectivity radius of $p$.
Let $\tilde{g}$ be the pulled back metric of M via $\phi$, and consider the map $\Psi_r: \bR^3\ra \bR^3: x
\mapsto r x$ and the rescaled metric $g := r^{-2}\Psi_r^*\tilde{g}$. Then for $0<r< \rho_{p}$ one has the following expansion of the area constrained Willmore equation \eqref{eq:2.20} on $S_1$: 
\begin{equation*}
\begin{split} 
  \mathcal{W}_{r,\lambda}(S_1)=&\ r^2\left(2\lambda -\frac{2}{3}\Scal\right) + 4 r^2 \Ric_{pq}x^px^q - r^3
\Scal_{,p}x^p+5r^3 \Ric_{pq,s}x^px^qx^s\\
& - r^4
\bigg(\frac{\lambda}{3}\Ric_{pq}+\frac{3}{5}\Scal_{,pq}+\frac{1}{5}\Delta\Ric_{pq}-\frac{2}{45}\Ric_p^{\
  k}\Ric_{kq}+\frac{4}{45}\Ric^{kl}\R_{kpql}\\
&\ +\frac{8}{45}\R^{kl\
  m}_{\ \ p}\R_{klqm}\bigg)x^px^q + r^4 \bigg(3
\Ric_{pq,st}+\frac{2}{3}\R^{k\ \ l}_{\ pq}\R_{kstl}\\
&\ -\frac{2}{3}\Ric_{pq}\Ric_{st}\bigg)x^px^qx^sx^t + O(r^5)\, .
\end{split}
\end{equation*}
\end{proposition}


%% file: equation.tex

\section{The equation}
\label{sec:equation}
In this section we prove theorem~\ref{thm:main} via the implicit
function theorem.  We consider a setup similar to Ye \cite{Ye:1991}. Let $(M,g)$ be given with injectivity radius
$\rho>0$. Fix a base point $p\in M$ and an orthonormal frame
$\{e_j\}_{j=1}^3$ for $T_p(M)$. Consider the map:
\begin{equation*}
  c : \bR^3 \supset B_\rho(0) \to M  : \tau \mapsto \exp_p(\tau),
\end{equation*}
where $\exp_p:T_p M \to M$ denotes the exponential map of $M$ at $p$.
Let $e_j^\tau$ be the parallel transports of the $e_j$ to $c(\tau)$
along the geodesic $t\mapsto c(t\tau)|_{t\in[0,1]}$. Define the map
\begin{equation*}
  F_\tau : \bR^3 \supset B_{\rho}(0) \to M : x \mapsto
  \exp_{c(\tau)} (x^i e_i^\tau).
\end{equation*}
Let
$\Omega_1 := \{ \varphi \in C^{4,\frac12}(S_1) \mid
\| \varphi \|_{C^{4,\frac12}(S_1)}< 1\}$
and for $\varphi\in \Omega_1$ let
$S_\varphi:=\{ (1+\varphi(x))x \mid x\in S_1\}$. For
$\tau\in B_\rho\subset\bR^3$ and $r\in(0,\rho/2)$ let
$S(r,\tau,\varphi)=F_\tau(\Psi_r (S_\varphi))$ where $\Psi_r$ denotes
scaling by $r$ as in section~\ref{sec:expansion}.  Define
\begin{equation*}
  \tilde \Phi : (0,\rho/2)\times B_\rho(0) \times \Omega_1 \times \bR \to
  C^{\frac12}(S_1) : (r,\tau,\varphi,\lambda) \mapsto 
  \tilde\Phi(r,\tau,\varphi,\lambda)
\end{equation*}
where $\tilde\Phi(r,\tau,\varphi,\lambda)$ is the function
\begin{equation}
  \label{eq:3}
  \Delta H + H|\Acirc|^2+H \Ric(\nu,\nu)+\lambda H
\end{equation}
evaluated on $S(r,\tau,\varphi)$ with respect to the metric $g$ and
pulled back to $S_1$ via the parameterization
$x \mapsto F_\tau(\Psi_r((1+\varphi(x))x))$.

Our goal is to find $r_0\in(0,\rho/2)$ and a map
\begin{equation*}
 (0,r_0) \to B_\rho \times \Omega_1 \times \R : r \mapsto (\tilde\tau(r),\tilde\varphi(r),\tilde\lambda(r))
\end{equation*}
so that
\[
\tilde\Phi(r,\tilde\tau(r),\tilde\varphi(r),\tilde\lambda(r))=0.
\]
Then for all $r\in(0,r_0)$ the surfaces $\Sigma_r:=S(r,\tilde\tau(r),\tilde\varphi(r))$
solve the equation
\begin{equation*}
  \Delta H +H|\Acirc|^2+H \Ric(\nu,\nu)+\tilde\lambda(r) H = 0 
\end{equation*}
as claimed. Up to reparameterization, the family
$(\Sigma_r)_{r\in(0,r_0)}$ is the family of solutions as in
Theorem~\ref{thm:main}, see Corollary~\ref{thm:reparametrize} for
details.

An equivalent way to define $\tilde\Phi(r,\tau,\varphi,\lambda)$ is to
evaluate the operator~\eqref{eq:3} on $S_\varphi$ with respect to the
metric $\tilde g^{r,\tau} := (\phi_\tau\circ \Psi_r)^* g$. To get a
uniform scale in $r$, we consider instead the rescaled metric
$g^{r,\tau} := r^{-2} \tilde g^{r,\tau}$ and define the rescaled
function $\Phi(r,\tau,\varphi,\lambda)$ to be the operator
\begin{equation}
  \label{eq:7}
  \Delta_{r,\tau} H_{r,\tau} + H_{r,\tau}|\Acirc_{r,\tau}|^2+H_{r,\tau} \Ric_{r,\tau}(\nu,\nu)+r^2 \lambda H_{r,\tau}
\end{equation}
evaluated on $S_\varphi$ with respect to $g^{r,\tau}$ and pulled back to
$S_1$ via the parameterization $x\mapsto (1+\varphi(x))x$ of
$S_\varphi$. From the scaling of the geometric quantities, we get
\begin{equation*}
  \Phi(r,\tau,\varphi,\lambda) = r^3 \tilde\Phi(r,\tau,\varphi,\lambda).
\end{equation*}
By definition
\begin{equation}
  \label{eq:6}
  \Phi(r,\tau,0,\lambda)=\mathcal{W}_{r,\lambda}(S_1),
\end{equation}
where $\mathcal{W}_{r,\lambda}(S_1)$ is from
Proposition~\ref{prop:expansion-Willmore} and the geometric quantities
in the expression for $\mathcal{W}_{r,\lambda}(S_1)$ are evaluated at $c(\tau)$.
Note that after shifting by $\tau$, the metric $g$
in Proposition~\ref{prop:expansion-Willmore} corresponds to
the metric $g^{r,\tau}$ here.

The linearization of the Willmore operator $\tilde\Phi$ is denoted by
$W_\lambda$. It was calculated in \cite[Section~3]{LMS:2011}. For a
variation of an arbitrary surface $\Sigma$ with normal speed $f$ it is
given by
\begin{equation}
  \label{eq:9}
  W_\lambda f=LLf+\frac12 \nabla^\star (H^2 \nabla f)-2 \nabla^\star(H\Acirc(\nabla f,\cdot))+\lambda Lf+fQ,
\end{equation}
where $\nabla^\star =- \text{div}$, $L=-\Delta -|A|^2-\Ric(\nu,\nu)$, and
\begin{equation}
  \label{eq:10}
  \begin{aligned}
    Q&= |\nabla H|^2 +2 \omega(\nabla H)+H\Delta H+2\langle \nabla^2 H ,\Acirc \rangle +2 H^2|\Acirc|^2+2H\langle \Acirc,T\rangle\\
    &\quad- H \nabla\Ric(\nu,\nu,\nu)-\frac12 H^2 |A|^2 + \frac12 H^2 \Ric(\nu,\nu).
  \end{aligned}
\end{equation}
Here $\omega = \Ric(\nu,\cdot)^T$ is the tangential projection of the $1$-form $\Ric(\nu,\cdot)$ to $\Sigma$ and $T = R(\cdot,\nu,\nu,\cdot)$.
All the geometric quantities in $W_\lambda$ are evaluated on $\Sigma$
with respect to the corresponding ambient geometry. For given
$f\in C^{4}(S_1)$ the family $t\mapsto S(r,\tau,t f)$ is a normal
variation of $S(r,\tau,0)$ with normal speed $rf$, so that
\begin{equation}
  \label{eq:4}
  \tilde\Phi_\varphi(r,\tau,0,\lambda) f = r W_\lambda f. 
\end{equation}
Here we evaluate $W_\lambda$ with respect to the metric $g$ in
$M$. Rescaling to the $g^{r,\tau}$ metric, we find that
\begin{equation}
  \label{eq:5}
  \Phi_\varphi(r,\tau,0,\lambda) f = r^4 W_\lambda f =
  W_{r,\tau,\lambda} f,
\end{equation}
where $W_{r,\tau,\lambda}$ is the linearized Willmore operator with
respect to $g^{r,\tau}$:
\begin{equation}
  \label{eq:8}
  \begin{aligned}
    W_{r,\tau,\lambda} f
    &= L_{r,\tau}L_{r,\tau} f +\frac12
    \nabla^\star_{r,\tau} (H_{r,\tau}^2\nabla_{r,\tau} f)
    \\
    &\quad
    -2\nabla^\star_{r,\tau}(H_{r,\tau}\Acirc_{r,\tau}(\nabla_{r,\tau}f,\cdot)+r^2\lambda L_{r,\tau}f+Q_{r,\tau}f.
  \end{aligned}
\end{equation}
Here we use the subscript $_{r,\tau}$ to denote quantities evaluated
with respect to the metric $g^{r,\tau}$.

In the limit $r\to 0$ the metric $g^{r,\tau}$ converges to the
Euclidean metric so that in the limit we have
\begin{equation*}
  W_{0,\tau,\lambda} f = L_0(L_0+2)f=(-\Delta) (-\Delta-2)f.
\end{equation*}
The kernel of this operator is given by
\begin{equation*}
  K:=\operatorname{ker} W_{0,\tau,\lambda} = \operatorname{Span} \{ 1,x^1,x^2,x^3\},
\end{equation*}
where the $x^i$ are the standard coordinate functions on $S_1$. We split this kernel into two parts:
\begin{equation}
  \label{eq:1}
    K_0 := \operatorname{Span} \{1\}\quad\text{and}\quad
    K_1 := \operatorname{Span} \{x_1,x_2,x_3\}. 	
\end{equation}
As in \cite{Ye:1991}, the function space $C^{4,\frac12}(S_1)$ splits
as a direct sum into $K$ and its $L^2$-orthogonal complement
$K^\perp$. It is standard to verify that we have the direct sum
decomposition of the target with respect to the $L^2$-scalar product:
\[
C^{0,\frac12}=  K + W_{0,\tau,\lambda}(K^\perp).
\]
Define the $L^2$-orthogonal projection maps
\begin{equation*}
  P_0:C^{0,\frac12}(S_1)\to K_0
  \quad\text{and}\quad
  P_1:C^{0,\frac12}(S_1)\to K_1.
\end{equation*}
The maps $T_0:K_0\to \bR$ and $T_1:K_1\to \bR^3$ identify $K_0$ and
$K_1$ with $\bR$ and $\bR^3$ according to the basis given in
equation~\eqref{eq:1}. Moreover, for $i\in \{0,1\}$ let
$\tilde P_i=T_i\circ P_i$. Denote by $\{e_1,e_2,e_3\}$
the standard basis of $\bR^3$.
\begin{lemma}
  \label{thm:projection-to-kernel}
  We have 
  \begin{align*}
    \tilde P_0 (\Phi(r,\tau,\varphi,\lambda))
    &=  8\pi r^2\left(\lambda+\frac13\Scal(c(\tau))\right)+O(r^4)
      +\tilde P_0 \left(\int_0^1 \Phi_{\varphi}(r,\tau,t\varphi,\lambda)\varphi\, dt\right)
    \intertext{and}
    \tilde P_1 (\Phi(r,\tau,\varphi,\lambda))&=
                                               \frac{4\pi}{3}r^3 \nabla_{e_i}\Scal(c(\tau)) e_i +O(r^5)
                                               +\tilde P_1 \left(\int_0^1 \Phi_{\varphi}(r,\tau,t\varphi,\lambda)\varphi\, dt\right).
  \end{align*}
\end{lemma}
\begin{proof}
  Start by writing
  \begin{align*}
    \Phi(r,\tau,\varphi,\lambda)&= \Phi(r,\tau,0,\lambda)+\int_0^1 \Phi_{\varphi}(r,\tau,t\varphi,\lambda)\varphi\, dt.
  \end{align*}
  By~\eqref{eq:6}, for $i\in \{0,1\}$ 
  \[
  \tilde P_i( \Phi(r,\tau,0,\lambda))=\tilde P_i (\mathcal{W}_{r,\lambda}(S_1)),
  \]
  Where $\mathcal{W}_{r,\lambda}(S_1)$ is evaluated at the base point
  $c(\tau)$. The right hand side can be calculated term by term from
  the expansion of $\mathcal{W}_{r,\lambda}(S_1)$ given in
  Proposition~\ref{prop:expansion-Willmore}:
  \begin{align*}
    \tilde P_0(\mathcal{W}_{r,\lambda}(S_1))&= 8\pi r^2\left(\lambda+\frac13 \Scal(c(\tau))\right)+O(r^4)\quad\text{and} \\
    \tilde P_1(\mathcal{W}_{r,\lambda}(S_1))&= \frac{4\pi}{3}r^3 \nabla_{e_i}\Scal(c(\tau))e_i+O(r^5).
  \end{align*}
  Note that all terms that contain an odd number of $x^i$-factors
  integrate to zero. For the other terms we used that
  $\int_{S_1} x^i x^p =\frac{4\pi}{3} \delta_{ip}$ and a similar
  expression for integrals involving four factors of components of
  $x$.
\end{proof}
\begin{lemma}
  \label{thm:derivativePhi1}
  For every $\tau\in\bR^3$ and every $\lambda\in\bR$ we have that
  \begin{equation*}
    \Phi_{\varphi r}(0,\tau,0,\lambda) = \left.\frac{\partial}{\partial r}\right|_{r=0}
    W_{r,\tau,\lambda} = 0.
  \end{equation*}
\end{lemma}
\begin{proof}
  For the proof, we have to calculate
  $\left.\frac{\partial}{\partial r}\right|_{r=0} W_{r,\tau,\lambda}$
  from its expression~\eqref{eq:8} taking into account its
  definition~\eqref{eq:9} and~\eqref{eq:10}. Since we compute
  $\frac{\partial}{\partial r} W_{r,\tau,\lambda}$ at $r=0$ we see
  that all terms that are product of at least two quantities that
  vanish at $(r,\varphi)=(0,0)$ do not contribute to the
  derivative. In particular
  \begin{equation}
    \label{eq:11}
    \begin{aligned}
      \left.\tfrac{\partial}{\partial r}\right|_{r=0} 
      \big(|\nabla H_{r,\tau} |^2 + 2\omega_{r,\tau}(\nabla H_{r,\tau}) + 2\langle
      \nabla^2 H_{r,\tau} ,\Acirc_{r,\tau} \rangle &
      \\ \qquad
      + 2 H_{r,\tau}^2|\Acirc_{r,\tau}|^2+2H_{r,\tau}\langle
      \Acirc_{r,\tau},T_{r,\tau}\rangle + r^2 \lambda L_{r,\tau} \big) & = 0.
    \end{aligned}
  \end{equation}
  From the proof of \cite[Lemma 1.3]{Ye:1991}, we quote equation
  (1.15)
  $\left.\tfrac{\partial}{\partial r}\right|_{r=0} g^{r,\tau} = 0$,
  its consequence
  $\left.\tfrac{\partial}{\partial r}\right|_{r=0} \Delta_{r,\tau} =
  0$,
  equation (1.17)
  $\left.\tfrac{\partial}{\partial r}\right|_{r=0} A_{r,\tau} = 0$,
  $\left.\tfrac{\partial}{\partial r}\right|_{r=0} \Ric_{r,\tau} = 0$,
  and assertion (1), that is
  $\left.\tfrac{\partial}{\partial r}\right|_{r=0} L_{r,\tau} = 0$.
  These identities also imply that
  $\left.\tfrac{\partial}{\partial r}\right|_{r=0} H_{r,\tau} = 0$
  and
  $\left.\tfrac{\partial}{\partial r}\right|_{r=0} \Acirc_{r,\tau} =
  0$.  From these formulas we find that
  \begin{equation*}
    \left.\tfrac{\partial}{\partial r}\right|_{r=0} 
      \big( L_{r,\tau} L_{r,\tau}  +
      2H_{r,\tau}\Delta_{r,\tau} H_{r,\tau} - \tfrac12
      H_{r,\tau}^2|A_{r,\tau}|^2 \big) = 0.
  \end{equation*}
  Since $\Ric_{r,\tau}=O(r^2)$ as in~\eqref{eq:12} and
  $\nabla_{r,\tau}\Ric_{r,\tau}(\nu_{r,\tau},\nu_{r,\tau},\nu_{r,\tau}) = O(r^3)$ by a similar
  argument, also
  \begin{equation*}
    \left.\tfrac{\partial}{\partial r}\right|_{r=0} \big(
    - H_{r,\tau} \nabla_{r,\tau}\Ric_{r,\tau}(\nu_{r,\tau},\nu_{r,\tau},\nu_{r,\tau}) +
    \tfrac12H_{r,\tau}^2 \Ric_{r,\tau}(\nu,\nu) \big) = 0.
  \end{equation*}
  To treat the final remaining terms in $W_{r,\tau,\lambda}$ rewrite it to
  \begin{equation}
    \label{eq:23}
    \begin{split}
      \tfrac12 \nabla_{r,\tau}^*(&H_{r,\tau}^2 \nabla_{r,\tau} f) - 2 \nabla^*_{r,\tau}(H_{r,\tau}
      \Acirc_{r,\tau} (\nabla_{r,\tau} f, \cdot))
      \\
      &
      =
      -H_{r,\tau}\langle \nabla H_{r,\tau}, \nabla_{r,\tau} f \rangle
      -\tfrac12 H_{r,\tau}^2 \Delta_{r,\tau} f
      + 2 \Acirc_{r,\tau}(\nabla H_{r,\tau},\nabla_{r,\tau} f) \\
      &\quad
      + 2 H_{r,\tau}\langle \nabla_{r,\tau}^* \Acirc_{r,\tau}, \nabla_{r,\tau}
      f\rangle + 2 H_{r,\tau} \langle \Acirc_{r,\tau},\nabla_{r,\tau}^2 f\rangle\\
      &
      =
      -\tfrac12 H_{r,\tau}^2 \Delta_{r,\tau} f + 2 \Acirc_{r,\tau}(\nabla
      H_{r,\tau},\nabla_{r,\tau} f) + 2 H_{r,\tau} \langle \Acirc_{r,\tau},\nabla_{r,\tau}^2 f\rangle\\
      &\ \ \  +
      2H_{r,\tau}\omega_{r,\tau}(\nabla_{r,\tau} f).
    \end{split}
  \end{equation}
  In the last equality we used the Codazzi equation in the form
  $-\nabla^* \Acirc = \tfrac12 \nabla H + \omega$. By inspection we
  see that each term in this expression has vanishing derivative in
  $r$-direction. Hence 
  \begin{equation*}
    \left.\frac{\partial}{\partial r}\right|_{r=0} W_{r,\tau,\lambda}f
    = 0
  \end{equation*}
  as claimed.
\end{proof}
Let $\varphi_0$ be the unique solution of the PDE
\begin{equation}
  \label{eq:13}
  W_{0,\tau,\lambda} \varphi_0=\left.\left(-\frac43 \Scal+4   \Ric_{pq}x^px^q\right)\right|_{r=0}.
\end{equation}
Note that the right hand side of this equation is an element of $K^\perp$ and hence $\varphi_0\in K^\perp$ is indeed uniquely defined. 
\begin{lemma}
  \label{thm:solution-lambda-tau}
  Let $(M,g)$ be a three dimensional manifold and $p\in M$ such that
  $\nabla\Scal(p) = 0$ and $\nabla^2 \Scal(p)$ is non-degenerate. For
  this base point there exists $r_0\in(0,\infty)$, an open
  neighborhood $U\subset C^{4,\frac12}(S_1)$ of $\varphi_0$, and
  functions
  \begin{equation*}
    \begin{aligned}
      \lambda : [0,r_0) \times U \to \bR : (r,\varphi) \mapsto \lambda(r,\varphi) \quad\text{and}\quad
      \tau : [0,r_0) \times U \to \bR^3  : (r,\varphi) \mapsto \tau(r,\varphi)
    \end{aligned}
  \end{equation*}
  so that
  \begin{align}
    \label{projection0}
    &\tilde P_i
    \big(\Phi(r,\tau(r,\varphi),r^2\varphi,\lambda(r,\varphi))\big)=0
    \quad\text{for}\quad i\in\{0,1\},
    \\
    & \tau(0,\varphi_0) = 0, \qquad\text{and}\qquad
      \lambda(0,\varphi_0) = -\frac13 \Scal(p).
  \end{align}
 \end{lemma}
\begin{proof}
  Calculate:
  \begin{align*}
    &\Phi_{\varphi}(r,\tau,tr^2\varphi,\lambda)
    \\
    &\quad=
      \Phi_\varphi(0,\tau,0,\lambda)
      + 
      r\int_0^1  \Phi_{\varphi r}(sr,\tau,str^2\varphi,\lambda)\, ds
      + tr^2 \int_0^1\Phi_{\varphi \varphi}(sr,\tau,str^2\varphi,\lambda)\varphi \, ds.
  \end{align*}
  Moreover, we have
  \begin{align*}
    r\int_0^1  \Phi_{\varphi r}(sr,\tau,str^2\varphi,\lambda)\, ds
    &=
      r^2 \int_0^1 \int_0^1 s\Phi_{\varphi r r}(usr,\tau,ustr^2 \varphi,\lambda)\, du \, ds\\
    &\quad
      + r^3\int_0^1 \int_0^1 st \Phi_{\varphi \varphi r} (usr,\tau,ustr^2\varphi,\lambda) \varphi \, du \, ds\\
    &\quad
      +r\Phi_{\varphi r}(0,\tau,0,\lambda).
  \end{align*}
  Hence 
  \begin{align*}
    r^2 \Phi_{\varphi}(r,\tau,tr^2\varphi,\lambda)
    &=
      r^2 \Phi_\varphi(0,\tau,0,\lambda)+r^3 \Phi_{\varphi r}(0,\tau,0,\lambda)\\
    &\quad
      +r^4  \int_0^1 t \Phi_{\varphi \varphi}(sr,\tau,str^2\varphi,\lambda)\varphi \, ds \\
    &\quad
      +r^4 \int_0^1 \int_0^1 s\Phi_{\varphi r r}(usr,\tau,ustr^2 \varphi,\lambda)\, du \, ds +O(r^5).
  \end{align*}
  By equation~\eqref{eq:5} we have
  $\Phi_\varphi(0,\tau,0,\lambda)= W_{0,\tau,\lambda}$ and
  from  Lemma~\ref{thm:derivativePhi1} we get
  $\Phi_{\varphi r}(0,\tau,0,\lambda)=0$. Therefore
  \begin{equation*}
    \tilde P_i\big(r^2
    \Phi_{\varphi}(r,\tau,tr^2\varphi,\lambda)\big)=O(r^4)
    \quad\text{for}\quad i\in \{0,1\}.
  \end{equation*}
 It follows from
  Lemma~\ref{thm:projection-to-kernel} that the
  system~\eqref{projection0} is equivalent to
  \begin{align*}
    8\pi\left(\lambda+\frac13 \Scal\right)
    &= -r^{-2} \tilde P_0 \left(\int_0^1 \Phi_{\varphi}(r,\tau,tr^2\varphi,\lambda)\varphi\, dt\right)+O(r^2) =O(r^2)\ \ \ \text{and}\\
    \frac{4\pi}{3} \nabla_{e_i}\Scal e_i
    &= r \tilde P_1 \left(\int_0^1 \int_0^1 t \Phi_{\varphi \varphi}(sr,\tau,str^2\varphi,\lambda)\varphi \varphi \, ds \, dt\right)\\
    &\quad +r \tilde P_1 \left(\int_0^1 \int_0^1 \int_0^1
      s\Phi_{\varphi r r}(usr,\tau,ustr^2 \varphi,\lambda)\varphi \,
      du \, ds\, dt\right) +O(r^2)
    \\
    &=O(r).
  \end{align*}
  By assumption $\nabla \Scal(p)=0$. Hence, at $r=0$,
  this system is satisfied for an arbitrary $\varphi_0\in K^\perp$, if  
  $\lambda|_{r=0}=-\frac13 \Scal(p)$ and $\tau|_{r=0} = 0$.

  The derivative with respect to $\lambda$ and $\tau$ at $r=0$ of the
  left hand side of this system is given by the matrix
  \[
  \begin{pmatrix}
    8\pi & \frac{8\pi}{3} \nabla \Scal|_{r=0} \\
    0 & \frac{4\pi}{3} \nabla^2 \Scal|_{r=0}
  \end{pmatrix}
  =
  \begin{pmatrix} 8\pi & 0 \\ 0 & \frac{4\pi}{3} \nabla^2 \Scal|_{r=0} \end{pmatrix}.
  \]
  By assumption $\nabla^2 \Scal|_{r=0}$ is non-degenerate. Hence, it
  follows from the implicit function theorem that there exist
  functions $\lambda=\lambda(r,\varphi)$ and $\tau=\tau(r,\varphi)$ as
  claimed at least for $(r,\varphi)$ in a neighborhood of
  $(0,\varphi_0)\in \bR\times C^{\frac12}(S_1)$.
\end{proof}

\begin{lemma}
  \label{thm:solution-final}
  Assume that $(M,g)$, $p$, $\phi_0$, $r_0$, $U$, $\lambda$ and $\tau$ are as in
  Lemma~\ref{thm:solution-lambda-tau}.

  Then there exists $r_1\in(0,r_0]$ and a function
  \begin{equation*}
    \varphi : [0,r_1) \to U : r\mapsto \varphi(r)
  \end{equation*}
  such that
  \begin{equation*}
    \begin{aligned}
      &\Phi(r,\tau(r,\varphi(r)),r^2\varphi(r), \lambda(r,\varphi(r)))
      = 0 \quad\text{and}      
    \end{aligned}
    \varphi(0) = \varphi_0.
  \end{equation*}
\end{lemma}
In particular, for small enough $r$, we have constructed a surface of
Willmore type with Lagrange multiplier $\lambda(r,\varphi(r))$. 
\begin{proof}
  Consider the expansion
  \begin{align*}
    \Phi(r,\tau,r^2\varphi,\lambda)
    &= \Phi(r,\tau,0,\lambda)+r^2 \int_0^1 \Phi_\varphi(r,\tau,tr^2 \varphi,\lambda)\, dt\\
    &= r^2\big(2\lambda -\frac{2}{3}\Scal+ 4\Ric_{pq}x^px^q\big)+O(r^3)+r^2 \Phi_\varphi(0,\tau,0,\lambda)\varphi\\
    &\quad
      +r^4 \int_0^1 \int_0^1 t\Phi_{\varphi \varphi} (sr,\tau,str^2\varphi,\lambda)\varphi \varphi \, dt\, ds\\
    &\quad
      +r^4\int_0^1 \int_0^1 \int_0^1 s \Phi_{\varphi r r}(usr,\tau,ustr^2 \varphi,\lambda)\varphi \, du\, ds\, dt\\
    &\quad
      +r^5\int_0^1 \int_0^1 \int_0^1 st \Phi_{\varphi \varphi r}(usr,\tau,ustr^2 \varphi,\lambda)\varphi \varphi \, du\, ds\, dt,
  \end{align*}
  where we used the fact that $\Phi_{\varphi r}(0,\tau,0,\lambda)=0$
  from Lemma~\ref{thm:derivativePhi1}.

  Since $\Phi_\varphi(0,\tau,0,\lambda)=W_{0,\tau,\lambda}$ as in equation~\eqref{eq:5},
  $\lambda(0,\varphi_0)=-\frac13 \Scal(p)$ and
  \begin{equation*}
    W_{0,\tau, \lambda} \varphi_0 =(-\frac43 \Scal+4
    \Ric_{pq}x^px^q)|_{r=0},
  \end{equation*}
  we conclude with the help of the implicit function theorem that,
  after dividing the above equation by $r^2$, there exists
  $r_1\in(0,r_0]$ and solution $\varphi:[0,r_1)\to U$ as claimed.
\end{proof}


%% file: foliation.tex
\section{The foliation}
\label{sec:foliation}
In this section we show that the surfaces $\Sigma_r$ indeed are a
foliation of a pointed neighborhood of $p\in M$. The method used is
very close to the arguments in \cite[pp. 390--391]{Ye:1991}. We start
with the following observation.
\begin{lemma}
  \label{thm:Phi-phi-r-r-is-even}
  The operator $\Phi_{\varphi r r} (0,\tau,0,\lambda)$ maps even
  functions to even functions.
\end{lemma}
\begin{proof}
  Note that  $\Phi_{\varphi r r} (0,\tau,0,\lambda)= \drr W_{r,\tau,\lambda}$
  where $W_{r,\tau,\lambda}$ is given by the expression in
  equation~\eqref{eq:8}. To prove the claim we check this expression
  term by term as in the proof of Lemma~\ref{thm:derivativePhi1}.
  
  We start by quoting from~\cite[Lemma 1.3]{Ye:1991} that $\drr
  L_{r,\tau}$ is an even operator. Hence, the claim follows from the
  facts that $\drr Q_{r,\tau}$ is an even function and in conjunction
  with equation~\eqref{eq:23} from the fact that the operator
  \begin{equation}
    \label{eq:14}
    f\mapsto
    -\tfrac12 H_{r,\tau}^2 \Delta_{r,\tau} f
    + 2 \Acirc_{r,\tau}(\nabla H_{r,\tau},\nabla_{r,\tau} f)
    + 2 H_{r,\tau} \langle \Acirc_{r,\tau},\nabla_{r,\tau}^2 f\rangle
    + 2 H_{r,\tau}\omega_{r,\tau}(\nabla_{r,\tau} f)
  \end{equation}
  maps even functions to even functions.

  To show this, we quote from the proof of~\cite[Lemma 1.3]{Ye:1991}
  that $\drr \Delta_{r,\tau}|_{r=0}$ is an even operator,
  $\drr\Ric_{r,\tau}(\nu_{r,\tau},\nu_{r,\tau})$ is an even
  function, $\drr |A_{r,\tau}|^2$ is an even function. Note that
  \cite[Equation (1.17)]{Ye:1991} implies that $\drr A_{r,\tau}$ is
  even, so that also $\drr H_{r,\tau}$ and $\drr \Acirc_{r,\tau}$ are even.

  Using these, it is easy to check that
  \begin{equation}
    \label{eq:22}
    \begin{aligned}
      &\drr\big(
      H_{r,\tau}\Delta_{r,\tau} H_{r,\tau}
      + 2\langle \nabla_{r,\tau}^2 H_{r,\tau} ,\Acirc_{r,\tau} \rangle 
      +2H_{r,\tau}\langle \Acirc_{r,\tau},T_{r,\tau}\rangle
      \\
      &\ \ \ \ 
      - H_{r,\tau} \nabla_{r,\tau} \Ric_{r,\tau}(\nu_{r,\tau},\nu_{r,\tau},\nu_{r,\tau})
      + H_{r,\tau}^2|\Acirc_{r,\tau}|^2
      + \frac12 H_{r,\tau}^2 \Ric_{r,\tau}(\nu_{r,\tau},\nu_{r,\tau})
      \big)
    \end{aligned}
  \end{equation}
  is even. For example consider (we omit the subscript $_{r,\tau}$ for
  clarity in the notation:
  \begin{equation*}
    \begin{aligned}
      \drr\!\! H\Delta H    
      &=\left(\drr H\right)\Delta H_{0,\tau} + \left((\dd{r} H) (\dd{r} \Delta H)\right)|_{r=0} + H_{0,\tau} \drr (\Delta H)
      \\
      &= H_{0,\tau} \left( 2(\dr\Delta) (\dr H) +  (\drr \Delta) H +
      \Delta_{0,\tau} \drr H \right)
      \\
      &= H_{0,\tau} (\drr \Delta_{r,\tau}) H_{0,\tau} + H_{0,\tau}
      \Delta_{0,\tau} \big( \drr H_{r,\tau} \big).
    \end{aligned}
  \end{equation*}
  In the second and third equality we used from the proof of
  Lemma~\ref{thm:derivativePhi1} that $\dr H_{r,\tau}=0$ and
  the fact that $H_{0,\tau}$ is constant. The right hand side is even,
  since $H_{0,\tau}$ is constant and thus even, since
  $\drr \Delta_{r,\tau}$ and $\Delta_{0,\tau}$ map even functions to
  even functions and since the product of even functions is even. The
  calculation for the other terms in~\eqref{eq:22} is similar. To
  treat the term
  $H_{r,\tau} \nabla_{r,\tau}
  \Ric_{r,\tau}(\nu_{r,\tau},\nu_{r,\tau},\nu_{r,\tau})$
  use that
  $\nabla_{r,\tau}\Ric_{r,\tau}(\nu_{r,\tau},\nu_{r,\tau},\nu_{r,\tau})=O(r^3)$.

  For the remaining terms in $Q_{r,\tau}$ note that the
  $\dr(\nabla_{r,\tau})$ is a first order differential operator that
  vanishes on constant functions. Hence,
  \begin{equation*}
    \begin{aligned}
      &\tfrac{1}{2}\drr |\nabla_{r,\tau} H_{r,\tau}|^2
      \\
      &
      =
      \left\langle\! \drr\!\!(\nabla_{r,\tau} H_{r,\tau}) ,\nabla_{0,\tau} H_{0,\tau} \right\rangle 
      +
      \left|(\dr\!\nabla_{r,\tau}) H_{0,\tau} + \nabla_{0,\tau} (\dr H_{r,\tau})\right|^2 = 0.
    \end{aligned}
  \end{equation*}
  This follows, since $H_{0,\tau}$ is constant,
  $\dr(\nabla_{r,\tau})= H_{0,\tau}$ and $\nabla_{0,\tau} H_{0,\tau}
  =0$ and since $\dr H_{r,\tau} =0$ as in
  Lemma~\ref{thm:derivativePhi1}. A similar computation yields that $\drr
  \omega_{r,\tau}(\nabla_{r,\tau} H) = 0$. We established that $\drr
  Q_{r,\tau}$ is an even function.

  It remains to consider the expression in~\eqref{eq:14}. Note that
  the first term is even by reasoning as above. The second term
  satisfies
  \begin{equation*}
    \drr \left( \Acirc_{r,\tau}(\nabla H_{r,\tau},\nabla_{r,\tau}
      f)\right) = 0.    
  \end{equation*}
  To treat the third term, use $\Acirc_{0,\tau}=0$ and
  $\dr\Acirc_{0,\tau} =0$ to compute
  \begin{equation*}
    \drr H_{r,\tau} \langle \Acirc_{r,\tau},\nabla_{r,\tau}^2 f\rangle
    =
    H_{0,\tau} \langle \drr \Acirc_{r,\tau} , \nabla_{0,\tau}^2 f\rangle.
  \end{equation*}
  Note that $\drr \Acirc_{r,\tau}$ is even and $\nabla_{0,\tau}^2$
  maps even functions to even functions so that this operator also has
  the desired property.

  For the last term from~\eqref{eq:14} we compute using
  $\omega_{0,\tau} =0$, $\dr\omega_{r,\tau}=0$ and $\dr H_{r,\tau} =0$
  that:
  \begin{equation*}
    \drr H_{r,\tau} \omega_{r,\tau} (\nabla_{r,\tau} f)
      =
      H_{0,\tau} \left(\drr\omega_{r,\tau}\right) \nabla_{0,\tau} f.
  \end{equation*}
  Note that $\nabla_{0,\tau}$ is the tangential gradient on $S^2$ and
  maps even functions to odd vector fields. Furthermore, by
  equation~\eqref{eq:12} and the fact that $\nu_{r,\tau} = x + O(r^2)$
  we have that
  \begin{equation*}
    \omega_{r,\tau} = r^2 \Ric_{c(\tau)} x + O(r^3) 
  \end{equation*}
  so that
  \begin{equation*}
    \drr \omega_{r,\tau} = \Ric_{c(\tau)} x
  \end{equation*}
  and hence $\drr \omega_{r,\tau}$ is an odd one form. Consequently
  $H_{0,\tau} \left(\drr\omega_{r,\tau}\right) \nabla_{0,\tau} f$ is
  an even function whenever $f$ is even. This concludes the proof.
\end{proof}
\begin{lemma}
  \label{thm:tau-is-o-r2}
  For $r\in(0,r_1)$ let
  $\Sigma_r := S(r,\tau(r,\varphi(r)),\varphi(r))$ be as in
  Lemma~\ref{thm:solution-final}.  Then  $\tau(r)=O(r^2)$ as $r\to 0$.
\end{lemma}
\begin{proof}
  It follows from the implicit function theorem that $\tau'(r)=O(r)$
  if and only if
  \begin{equation}
    \label{eq:21}
    \tilde P_1 (  \Phi_{\varphi \varphi}(0,\tau,0,\lambda)\varphi_0
    \varphi_0)=0
    \quad\text{and}\quad
    \tilde P_1 (\Phi_{\varphi r r} (0,\tau,0,\lambda)\varphi_0)=0.
  \end{equation}
  To establish the first identity, note that by the fact that
  equation~\eqref{eq:13} has unique solutions and since
  $W_{0,\tau,\lambda}$ is invariant under the reflection at the
  origin, it follows that $\varphi_0$ is an even function.

  Furthermore, for every $t$ in a neighborhood of $0$ the euclidean
  Willmore operator $\Phi(0,\tau,t\varphi_0,\lambda)$ evaluates to an
  even function. Hence also
  \begin{equation*}
    \Phi_{\varphi \varphi}(0,\tau,0,\lambda)\varphi_0 \varphi_0
    =
    \left.\frac{\partial^2}{\partial t^2}\right|_{t=0} \Phi(0,\tau,t\varphi_0,\lambda)
  \end{equation*}
  is even. Since $\tilde P_1$ vanishes on even functions, the first
  claim from \eqref{eq:21} follows.

  To prove the second identity, note that by
  Lemma~\ref{thm:Phi-phi-r-r-is-even} the operator
  $\Phi_{\varphi r r} (0,\tau,0,\lambda)$ maps even function to even
  functions and the claim follows in a similar manner.
\end{proof}

This lemma implies in particular that we can reparameterize the
solutions that we found in section~\ref{sec:equation} by their area.
\begin{corollary}
  \label{thm:reparametrize}
  For $r\in(0,r_1)$ let
  $\Sigma_r := S(r,\tau(r,\varphi(r)),r^2\varphi(r))$ be as in
  Lemma~\ref{thm:solution-final}. Consider the area of $\Sigma_r$ in
  $(M,g)$ as a function of $r$:
  \begin{equation*}
    a :(0,r_1) \to (0,\infty) : r \mapsto \int_{\Sigma_r} 1 \dmu_g.
  \end{equation*}
  Then there exists $r_2\in(0,r_1]$ so that $a$ is strictly increasing
  on $(0,r_2)$. In particular:
  \begin{equation*}
    a(r) = 4\pi r^2 + O(r^4)
    \quad\text{and}\quad
    a'(r) = 8\pi r + O(r^3).
  \end{equation*}
\end{corollary}

\begin{proof}
  Note that $a$ extends as a smooth function to $r=0$ so that $a(0)=0$
  and hence the first claim follows from the second. We first note that
  \begin{equation*}
    a'(r) = - \int_{\Sigma_r} g(\vec{H}, X) \dmu_g\ ,
  \end{equation*}
  where $X$ is the variation vector-field along this family. Note that
  $X$ is not unique, whereas $X^\perp$ is well defined. Recall that
  from Lemma \ref{thm:solution-final} we have that $\Sigma_r$ is an
  exponential normal graph over $S_r(\tau(r))$ with height function
  $r^3\varphi(r)$ such that $\varphi(r) \rightarrow \varphi_0$ as
  $r \rightarrow 0$. Furthermore, by Lemma \ref{thm:tau-is-o-r2} we
  have that $\tau(r) = O(r^2)$ as $r \to 0$. This implies that
  $$ X^\perp\big|_{\Sigma_r} = \frac{\partial}{\partial r_\tau} + O(r^2)$$
  where $r_\tau = d_g(\tau(r), \cdot)$. Furthermore, by the above and
  \eqref{eq:2.8} we have that
  $$ H_{\Sigma_r} = H_{S_r(\tau(r))} + O(r^2) = \frac{2}{r} + O(r)\, ,$$
  as well as
  $$ \nu_{\Sigma_r} = \nu_{S_r(\tau(r))} + O(r^3)\, .$$
  Also note that from \eqref{eq:2.4} we have 
$$ \int_{S_r(\tau(r))} 1 \dmu_g = 4\pi r^2 + O(r^4)\ .$$
This implies
  \begin{equation*}
    a'(r) = - \int_{\Sigma_r} g(\vec{H}, X) \dmu_g =  8\pi r + O(r^3)\, .
  \end{equation*}
\end{proof}
Due to Corollary~\ref{thm:reparametrize} there exists $a_0\in
(0,\infty)$ and a map $\tilde r : (0,a_0) \to (0,r_2)$ such that
$|\Sigma_{\tilde r(a)}| = a$. We slightly abuse notation by letting
\begin{equation}
  \label{eq:24}
  \Sigma_a := \Sigma_{\tilde r(a)} \quad\text{for}\quad a\in(0,a_0)
\end{equation}
This finishes the existence part of the proof of
Theorem~\ref{thm:main}. To complete the proof, it remains to show the
following:
\begin{proposition}
  \label{thm:is-foliation}
  For $r\in(0,r_1)$ let
  $\Sigma_r := S(r,\tau(r,\varphi(r)),\varphi(r))$ be the surfaces from
  Lemma~\ref{thm:solution-final}. Then there exist $r_2\in(0,r_1]$ so
  that the family $\{\Sigma_r\}_{r\in(0,r_1)}$ is a foliation of a
  pointed neighborhood of $p$.
\end{proposition}
\begin{proof}
  Define the maps
  \begin{align*}
    \Psi^r &:= \exp_p^{-1} \exp_{c(\tau(r))},\\
    \Psi(r,x) &:= \Psi^r(r(x+r^2\varphi(r)(x)))	\quad \text{and}\\
    \beta(r,x)&:=\frac{\Psi(r,x)}{|\Psi(r,x)|}.
  \end{align*}
  We claim that there exists $\tilde r \in(0,r_1]$ such that
  $|\Psi(r,x)|\not=0$ every $x\in S^2$ and such that
  $\beta(r,\cdot) :S^2 \to S^2 $ is a family of diffeomorphisms which
  can be smoothly extended to $r=0$ by the identity.

  Indeed, this follows from the facts that 
  \begin{equation*}
    \frac{\partial \Psi}{\partial r}
    = (d_x\Psi^r)\big(x+r^2\varphi(r)(x)+r(r^2\varphi(r)(x))_r\big)
    +\left(\frac{\partial \Psi^r}{\partial r}\right)\big(r(x+r^2\varphi(r)(x))\big)
  \end{equation*}
  and 
  \begin{equation*}
    \left.\frac{\partial \Psi^r}{\partial r}\right|_{r=0}
    = \left.\frac{\partial}{\partial\tau^i}\big(\exp_p^{-1} \exp_{c(\tau(r))}\big)\right|_{\tau=0} 
    \cdot \left.\frac{\partial \tau^i}{\partial r}\right|_{r=0}=0
  \end{equation*}
  where we used Lemma~\ref{thm:tau-is-o-r2} in the last equality. In combination
  \begin{equation*}
    \frac{\partial \Psi}{\partial r}(0,x)=x.
  \end{equation*}
  Hence
  \begin{equation*}
    \Psi(r,x)=rx+O(r^2)\quad\text{as}\quad r\to 0.
  \end{equation*}
  In particular, $\Psi(r,x)\not=0$ for $r$ small enough and
  \begin{equation*}
    \beta(r,x)=\frac{x+O(r)}{|x+O(r)|}.
  \end{equation*}
  This establishes the claim.

  Let $\eta(r,x):= |\Psi(r,\beta^{-1}(r,x))|$ and calculate
  \begin{equation*}
    \frac{\partial \eta}{\partial r}
    = \frac{\Psi}{|\Psi|} \left(\frac{\partial \Psi}{\partial r}+(d_x\Psi)\left(\frac{\partial}{\partial r}\beta^{-1}\right)\right)
    = \frac{x+O(r)}{|x+O(r)|}\, (x+O(r)).
  \end{equation*}
  This yields
  \begin{equation*}
    \left.\frac{\partial \eta}{\partial r}\right|_{r=0}=1.
  \end{equation*}
  Consequently $\eta$ is strictly increasing for $r$ small enough
  which shows that all the surfaces are disjoint.
\end{proof}

%% file: uniqueness.tex
\section{Local Uniqueness}
\label{sec:uniqueness}
By inspecting the proof of theorem~\ref{thm:main} above and by the
local uniqueness of solutions obtained via the implicit function
theorem, we obtain a local uniqueness result for the $\Sigma_a$. To
state this, we use the notation introduced at the beginning of
section~\ref{sec:equation}.
\begin{corollary}
  \label{thm:uniqueness1}
  Fix $\alpha\in(0,1)$. Let $(M,g)$ be a Riemannian 3-manifold and
  $p\in M$ be a non-degenerate critical point of the scalar curvature.

  Denote by $\varphi_0\in C^\infty(S^2)$ the solution of~\eqref{eq:13}
  where $\Ric$ is evaluated at $p$.

  Then there exist $r_0\in(0,\infty)$, a neigborhood
  $\Omega\subset C^{4,\alpha}(S^2)$ of $\varphi_0$, a neighborhood
  $U\subset\bR^3$ of the origin, and an open interval $I\subset\bR$
  such that $-\frac{1}{3}\Scal(p) \in I$ with the following
  properties:

  Assume that $\Sigma\subset M$ is such that:
  \begin{enumerate}
  \item $\Sigma = S(r,\tau,r^2\varphi)$ up to reparameterization,
  \item On $\Sigma$ we have that
    \begin{equation*}
       \Delta H + H|\Acirc|^2 + H \Ric(\nu,\nu) = \lambda H.
    \end{equation*}
  \item $(r,\tau,\varphi,\lambda) \in (0,r_0) \times U \times
    \Omega \times I$ and $P_i(\varphi)=0$ for $i\in\{0,1\}$.
  \end{enumerate}
  Then $\Sigma = \Sigma_a$ where $\Sigma_a$ is as in
  equation~\eqref{eq:24} and such that $|\Sigma|=|\Sigma_a|$.
\end{corollary}

Note that if $\Omega_b\subset C^{4,\alpha}(S^2)$ is any bounded subset
and $\varphi \in\Omega_b$, then there exists a constant
$C=C(\Omega_b)$ such that
\begin{equation*}
  \left| \CW(S(r,\tau,r^2\varphi) - 4\pi \right| \leq C r^2.
\end{equation*}
In \cite{LM:2013} it was shown that there exists $\eps=\eps(M,g)>0$
such that if $\Sigma$ is a solution~\eqref{eq:maineq} with
$|\Sigma|<\eps$ and $\CW(\Sigma)\leq 4\pi+\eps$ also satisfies
\begin{equation*}
  \left|\lambda + \frac{1}{3}\Scal(p)\right| \leq Cr
\end{equation*}
for some constant $C=C(M,g)$.  Hence, it follows that in the statement
of Corollary~\ref{thm:uniqueness1} the condition on $\lambda$ is in
fact not needed.
\begin{corollary}
  \label{thm:uniqueness2}
  There exist $r_0'\in(0,\infty)$, a neigborhood
  $\Omega'\subset C^{4,\alpha}(S^2)$ of $\varphi_0$, a neighborhood
  $U'\subset\bR^3$ of the origin with the following
  properties:

  Assume that $\Sigma\subset M$ is such that:
  \begin{enumerate}
  \item $\Sigma = S(r,\tau,r^2\varphi)$ up to reparameterization,
  \item On $\Sigma$ we have that
    \begin{equation*}
       \Delta H + H|\Acirc|^2 + H \Ric(\nu,\nu) = \lambda H.
    \end{equation*}
  \item $(r,\tau,\varphi) \in (0,r_0) \times U \times (\Omega\cap K)$
    and $P_i(\varphi)=0$ for $i\in\{0,1\}$.
  \end{enumerate}
  Then $\Sigma = \Sigma_a$ where $\Sigma_a$ is as in
  equation~\eqref{eq:24} and such that $|\Sigma|=|\Sigma_a|$.
\end{corollary}

Note that this uniqueness applies to individual solutions
of~\eqref{eq:maineq} and not to whole foliations. It is not difficult
though, to prove a result similar to~\cite[Section 2]{Ye:1991} to deal
with the uniqueness of foliations centered at $p$ based on the a
priori estimates on such surfaces in~\cite{LM:2010,LM:2013,Laurin-Mondino:2014}.
